\documentclass[10pt, reqno]{amsart}%
\usepackage{amssymb}
\usepackage{amsmath,esint}
\usepackage{amsthm,mathrsfs, dsfont}
\usepackage{amscd,mdwlist}
\usepackage{polynom}
\usepackage[margin=0.8in]{geometry}
\usepackage{color}
\usepackage{cite,todonotes}
\usepackage{bbm}
\usepackage{tcolorbox}
\usepackage{enumerate} 
\usepackage[pagebackref, 
pdfpagelabels, plainpages=false
]{hyperref}

\hypersetup{
colorlinks=true,
linkcolor=red,
filecolor=magenta,
urlcolor=cyan,
citecolor=blue,
}

\theoremstyle{plain}
\newtheorem{theorem}{Theorem}[section]

\theoremstyle{definition}

\newtheorem*{theorem*}{Theorem}
\newtheorem{innercustomgeneric}{\customgenericname}
\providecommand{\customgenericname}{}
\newcommand{\newcustomtheorem}[2]{%
\newenvironment{#1}[1]
{%
\renewcommand\customgenericname{#2}%
\renewcommand\theinnercustomgeneric{##1}%
\innercustomgeneric
}
{\endinnercustomgeneric}
}

\newcustomtheorem{customthm}{Theorem}

\newtheorem*{proposition*}{Proposition}

\newcustomtheorem{customprop}{Proposition}

\newtheorem*{remark*}{Remark}

\newcustomtheorem{customrmk}{Remark}
\newtheorem*{corollary*}{Corollary}

\newcustomtheorem{customcor}{Corollary}

\DeclareMathOperator{\R}{\mathbb{R}}
\renewcommand{\d}{\mathrm{d}}

\newcommand{\eps}{\varepsilon}

\newcommand{\vertiii}[1]{{\left\vert\kern-0.25ex\left\vert\kern-0.25ex\left\vert #1 \right\vert\kern-0.25ex\right\vert\kern-0.25ex\right\vert}}
\usepackage[english]{babel}

\addto\extrasenglish{%

}

\begin{document}
\title{Banach-Saks Theorem for $L^1$ revisited}

\author[Foghem]{Guy Foghem$^{\dagger}$}


\address{ {\small Fakult\"{a}t f\"{u}r Mathematik Institut f\"{u}r Analysis, TU Dresden Zellescher Weg 23/25, 01217, Dresden, Germany. Email: guy.foghem[at]tu-dresden.de}}

\thanks{$^{\dagger}$The author is supported by the Deutsche Forschungsgemeinschaft/German Research Foundation (DFG) via the Research Group 3013: "Vector-and Tensor-Valued Surface PDEs"}

\begin{abstract}
The  Banach-Saks property is an important tool in analysis with applications ranging from partial differential equations (PDEs) to calculus of variations and probability theory. We survey the Banach-Saks property for $L^p$-spaces, with a particular emphasis on the case where $p=1$. In other words, we revisit the celebrated result by W. Szlenk (1965) in a more general context, demonstrating that $L^1$-spaces possess the weak Banach-Saks property.
\end{abstract}

\makeatletter
\@namedef{subjclassname@2020}{\textup{2020} Mathematics Subject Classification}
\makeatother


\maketitle 
\vspace{-5mm}

\section{Introduction}
A Banach space $X$ satisfies the Banach-Saks property (resp. the weak Banach-Saks property) if  every bounded sequence $(x_n)_n\subset X$  (resp. weakly converging to $x$ in $X$) admits a subsequence $(x_{n_j})_j$ strongly converging in the Ces\'{a}ro sense, that is,  $\|\overline{x}_{n_j}-x\|_X\xrightarrow{j\to\infty}0$, $x\in X$ where 
\vspace{-2mm}
\begin{align*}
\overline{x}_{n_j}=\frac{1}{j}\sum_{k=1}^j x_{n_k}. 
\end{align*}
This variation in the definition irrelevant  if $X$ is reflexive. Namely, in a reflexive Banach space the Banach-Saks property and the weak Banach-Saks property are equivalent. For a general Banach space, because weak converging sequence are bounded,  the weak Banach-Saks property is implied by the Banach-Saks property but the converse is  not always true. Notable examples of Banach spaces not satisfying the Banach-Saks property include $L^1-$spaces, this is because they are not reflexive.  Indeed, a result of T. Nishiura and D. Waterman \cite{NiWa63} asserts that a Banach space satisfying the  Banach-Saks property is automatically reflexive (interestingly, the converse is not true, as constructions of reflexive Banach spaces not satisfying the Banach-Saks property are provided by B. Beauzamy and A. Baernstein  in \cite{Bae72, Bea79}).  However, it was recognized by Szlenk \cite{Szl65} that $L^1(0,1)$ enjoys the weak Banach-Saks property. It turns out that $L^1-$spaces are perfect examples of a non-reflexive Banach space satisfying the weak Banach-Saks property. The aim of this note is to address the weak Banach Saks property of $L^1-$spaces in the general context. From now on, we write $L^p(X)$, $1\leq p\leq\infty$ in the sequel to tacitly denote the usual Lebesgue spaces  associated with on a measure space $(X,\mathcal{A}, \mu)$, i.e., $\mathcal{A}$ is a $\sigma$-algebra on a set $X$ and $\mu$ is a positive measure on $\mathcal{A}$.  We say that a sequence $(u_n)_n\subset L^1(X)$ converges weakly to $u$ in $L^1(X)$  and we write $u_{n}\rightharpoonup u$ if 
\begin{align*}
( v, u_n-u)_{(L^1(X))', L^1(X)}\xrightarrow{n\to\infty}0 
\quad \text{for all $v\in (L^1(X))'$},
\end{align*}
where $(\cdot, \cdot)_{(L^1(X))', L^1(X)}$ is the dual paring  between $L^1(X)$ and its dual $(L^1(X))'$. In passing, we recall the well-known fact from the Riesz representation for $L^1(X)$ 
(see for instance \cite[Corollary 2.41 \& Remark 2.42]{FoLe07} or the more recent proof in \cite{Shi18}), viz., if $\mu$ is $\sigma$-finite then we can identify $(L^1(X))'\equiv L^\infty(X)$. In the latter case, the aforementioned weak convergence $L^1(X)$ boils down to the following condition 
\begin{align*}
\int_X u_n(x) v(x)\d \mu (x)\xrightarrow{n\to\infty}	\int_X u(x) v(x)\d \mu (x)\quad \text{for all $v\in L^\infty(X)$}. 
\end{align*}
\begin{theorem}\label{thm:banach-saks-l1} The space $L^1(X)$ enjoys the weak  Banach-Saks property, i.e.,  for any sequence  $u_n\rightharpoonup u$ weakly in $ L^1(X)$, there is a subsequence $(u_{n_j})_j$ such that  $\|\overline{u}_{n_j}-u\|_{L^1(X)}\xrightarrow{j\to\infty}0$. 
\end{theorem}
\noindent Interestingly, Theorem \ref{thm:banach-saks-l1} is reminiscent of the renowned result by W. Szlenk \cite{Szl65}, who originally established Theorem \ref{thm:banach-saks-l1} for the space $L^1(0,1)$ endowed with the Lebesgue measure. Notably, the argument of W. Szlenk \cite{Szl65} carries out to the space $L^1(X)$ when the measure $\mu$ is finite, i.e., $\mu(X)<\infty$. It is worth emphasizing, however, that we do not impose any restrictions on the measure $\mu$. In fact, Theorem \ref{thm:banach-saks-l1} is a straightforward consequence of the following Theorem \ref{thm:uniform-weak-banach-saks} which is more general. 
\begin{theorem}\label{thm:uniform-weak-banach-saks}
Assume $u_n\rightharpoonup u$  weakly in $L^1(X)$. There is a subsequence $(u_{n_j})_j$ such that 
\begin{align*}
\sup_{\stackrel{\theta: \mathbb{N}\to\mathbb{N},\, s.t. }{\theta(\tau)<\theta(\tau+1)}} \Big\| \frac{1}{j} \sum_{k=1}^j u_{n_{\theta(k)}}-u\Big\|_{L^1(X)}\xrightarrow{j\to\infty}0. 
\end{align*}
The supremum is performed over all strictly increasing mapping $\theta\,:\mathbb{N}\to \mathbb{N}$. 
\end{theorem}
Our proof of the theorem \ref{thm:uniform-weak-banach-saks} is based on the Dunford-Pettis characterization of the weak convergence in $L^1(X)$  and a refinement of the arguments of Szlenk's  proof \cite{Szl65}. Let us comment on some related works in the literature. The Banach-Saks phenomenon was first established by S. Banach and S. Saks \cite{BaSa30} for the space $L^p(0, 1)$, $1<p<\infty$.  For the convenience of the reader, we present their proof for  $L^p(X)$, $1<p<\infty$, in Appendix \ref{sec:Appendix} (Theorem \ref{thm:banach-saks-lp}). This  result was subsequently extended to a uniform convex space by S. Kakutani \cite{Kak39} (see the proof in \cite[P.124]{Die84}) and N. Okada \cite{Oka84} who proved that a Banach space  whose dual  is uniformly convex also features  the Banach-Saks property (see Theorem \ref{thm:banach-saks-Okada} and its proof below). As mentioned earlier, due to the lack of reflexivity, the Banach-Saks property fails in general\footnote{In some pathological cases $L^1(X)$ might be reflexive and enjoy the Banach-Saks property as well. A blatant instance is obtained by considering $L^1(X_d, \mu)\equiv \R^d$ where $\mu$ is the counting measure on $X_d=\{1,2,\cdots,d\}$ and obviously $\|u\|_{L^1(X_d,\mu)}=\sum_{i=1}^{d}|u(i)|$.} for the space $L^1(X)$, but it was shown by W. Szlenk \cite{Szl65} that $L^1(0,1)$ rather satisfies the weak Banach-Saks property. As a matter of fact, an intriguing anecdote concerning the Banach-Saks phenomenon is related to the original work of Banach and Saks \cite{BaSa30}. Indeed, Banach and Saks claimed the failure of the weak Banach property for $L^1(0,1)$ and also claimed to have generated a weakly null sequence in $L^1(0,1)$  without any subsequences having strongly converging in the Ces\'aro sense. Later, the proof  of W. Szlenk \cite{Szl65} however, revealed the error in the assertion of  Banach and Saks \cite{BaSa30}.  For the sake of completeness, it is important to mention that in the case $p=\infty$, even the weak Banach-Saks property fails in general for the space $L^\infty(X)$ and especially for  the space $C[0,1]$. This was first established by J. Schreier  in \cite{Sch30} and later extended by  N. Farnum in \cite{Far74} for general spaces $C(S)$ where $S$ is a metric space. The result by W. Szlenk \cite{Szl65} sparkled significant interests in the area of probability theory, where one sometime  wishes to have the pointwise convergence almost everywhere of random variables instead of strong convergence.  The first step in this direction, attributed to J. Koml\'{o}s \cite{Kom67} (see a recent proof in \cite[Theorem 4.7.24]{Bog07}) infers  that \textit{if $\mu(X)<\infty $, then  a bounded sequence $(u_n)_n$ in  $L^1(X)$  admits a subsequence  $(u_{n_j})_j$ and $u\in L^1(X)$ such that  for all strictly increasing mapping $\theta\,:\mathbb{N}\to \mathbb{N}$,  the sequence $(u_{n_{\theta(j)}})_j$ converges  to $u$ almost everywhere in the Cesàro sense, that is, $\overline{u}_{n_{\theta(j)}}\to u$ almost everywhere in $X$}. This result was improved  by D. Aldous in \cite{Ald77}.  Much later, I. Berkes \cite{Berk90} extended the result of J.  Koml\'{o}s \cite{Kom67} and D. Aldous in \cite{Ald77}  to the space $L^p(X)$, $1\leq p<\infty$ with $\mu(X)<\infty$; see  \cite[Theorem 29, P. 102]{Woj91} for a detailed proof.  
Last but certainly not least, the weak Banach-Saks property represents an enhancement of the sequential Mazu's  lemma \cite[P. 6]{ET76}, which asserts that any weakly convergent sequence in a normed space admits a sequence of convex combinations of its members that converges strongly to the same limit. However, the major limitation of this result is that, because it uses the Hahn-Banach theorem, the convex combinations are not explicitly determined.

\vspace{-2mm}

\section{Proof of the main result}
\noindent Analogously to Theorem \ref{thm:uniform-weak-banach-saks},  Hilbert spaces satisfy a stronger notion  called the uniform Banach-Saks property. The proof is adapted from those of  \cite{Szl65,RiSz90}.   
\begin{theorem}\label{thm:banach-saks-hilbert}
A Hilbert space $(H, (\cdot,\cdot)_H)$ satisfies the uniform Banach-Saks property, i.e., every bounded sequence $(x_n)_n\subset H$ admits a subsequence $(x_{n_j})_j$ and $x\in H $ such that 
\begin{align*}
\lim_{j\to\infty}	\sup_{\stackrel{\theta: \mathbb{N}\to\mathbb{N},\, s.t. }{\theta(\tau)<\theta(\tau+1)}} \Big\| \frac{1}{j} \sum_{k=1}^j x_{n_{\theta(k)}}-x\Big\|_{H}=0.
\end{align*}
In particular, we have $\| \overline{x}_{n_{j}}-x\|_{H}\xrightarrow{j\to\infty}0.$
\end{theorem}
\begin{proof}
A bounded sequence $(x_n)_n\subset H$  say $\sup_{n\geq1}\|x_n\|_H\leq r$, for some $r>0$, has a weak converging subsequence. Without loss of generality, we assume that  $(x_n)_n\subset H$ weakly converges to $x$ in $H$  and that $x=0$. Put $x_{n_1}=x_1$ assume $x_{n_{j-1}}$ is given, $j\geq2$. Since $(x_n)_n$ converges weakly to $x=0$, we choose $n_j>n_{j-1}$ such that 
\begin{align*}
|(x_{n_k}, x_{n_j})_H|\leq \frac{1}{j+1}\quad \text{for every $k=1,2,\cdots, j-1$}.
\end{align*}
A strictly increasing $\theta:\mathbb{N}\to\mathbb{N}$ satisfies $\theta(j)\geq j$. By construction, the sequence $(x_{n_{\theta(j)}})_j$ satisfies 

\begin{align*}
\Big\|\sum_{k=1}^jx_{n_{\theta(k)}}\Big\|^2_H
&= \sum_{k=1}^j\|x_{n_{\theta(k)}}\|^2_H+ 2  \sum_{i=2}^{j} \sum_{k=1}^{i-1}(x_{n_{\theta(k)}}, x_{n_{\theta(i)}} )_H\\
&\leq jr^2+ 2\sum_{i=2}^{j}\frac{i-1}{\theta(i)+1}\leq jr^2+ 2j.
\end{align*}
Finally, the sought result follows since we obtain 
\begin{align*}
\sup_{\stackrel{\theta: \mathbb{N}\to\mathbb{N},\, s.t. }{\theta(\tau)<\theta(\tau+1)}} \Big\|\frac{1}{j}\sum_{k=1}^jx_{n_{\theta(k)}}\Big\|^2_H \leq \frac{r^2+ 2}{j}\xrightarrow{j\to\infty}0.
\end{align*}
\end{proof}
\vspace{-2mm}
\noindent Next, we need the Dunford-Pettis criterion for weak compactness in $L^1(X)$. This criterion was originally established by N. Dunford and B. Pettis in \cite{Dun39, DuPe40}. A more contemporary version of the Dunford-Pettis theorem, credited to L. Ambrosio, N. Fusco and D. Pallara \cite{AFP00} with a meticulous proof can be found in \cite[Theorem 2.54]{FoLe07}; see  also the versions in \cite[Theorem 4.7.18 \&  4.7.20]{Bog07}. To facilitate the statement of the result, it is convenient to recall the notions of tightness and uniform integrability. 
Let $\mathcal{F}\subset L^1(X)$ be a subset. The set  $\mathcal{F}$ is uniformly integrable (or equiintegrable) if
\begin{align*}
\lim_{\mu(E)\to 0} \sup_{u\in \mathcal{F}} \int_E |u(x)|\,\d \mu(x)= 0.
\end{align*}
That is, to be strict, for every $\varepsilon>0$ there is $\delta>0$ such that for a measurable set $E\in \mathcal{A}$ with  $\mu(E)<\delta$, 
\begin{align*}
\int_E |u(x)|\,\d \mu(x)<\varepsilon\qquad\textrm{for all}\quad u\in \mathcal{F}.
\end{align*}
The  set $\mathcal{F}$ is tight if 
\begin{align*}
\inf_{\mu(E) <\infty} \sup_{u\in \mathcal{F}} \int_{X\setminus E} |u(x)|\,\d \mu(x)= 0.
\end{align*}
That is, for every $\varepsilon>0$ there exists a measurable set $E$ such that $0< \mu(E)<\infty$ and 
\begin{align*}
\int_{X\setminus E} |u(x)|\,\d \mu(x)<\varepsilon\qquad\textrm{for all}\quad u \in \mathcal{F}.
\end{align*}

\begin{theorem}[Dunford-Pettis]\label{thm:dunford-petit}
For sequence $(u_n)_n\subset L^1(X)$, the following assertions are equivalent. 
\begin{enumerate}[$(i)$]
\item The sequence $(u_n)_n$ is relatively weakly compact in $L^1(X)$. 
\item The sequence $(u_n)_n$ is bounded in $L^1(X)$, uniformly integrable and tight. 
\end{enumerate}
\end{theorem}
\noindent A fundamental consequence of Theorem \ref{thm:dunford-petit} is that (see \cite[Corollary 2.58]{FoLe07}) a bounded sequence $(u_n)_n\subset L^1(X)$ weakly converges to $u\in L^1(X)$ if and only if  
\begin{align*}
\int_{A}u_n(x)\d\mu(x)\xrightarrow{n\to\infty}\int_{A}u(x)\d\mu(x)\quad\text{for every measurable set $A\in \mathcal{A}$}.  
\end{align*}
In order to proof the main Theorem \ref{thm:uniform-weak-banach-saks} we need the following ancillary result. 
\begin{theorem}\label{thm:eps-uniform-cesaro}
Let $u_n\rightharpoonup 0$ in $L^1(X)$ then for $\eps>0$ there is a subsequence $(n_{\eps, j})_j \equiv (n_j)_j$ such that 
\begin{align*}
\limsup_{j\to \infty}\sup_{\stackrel{\theta: \mathbb{N}\to\mathbb{N},\, s.t. }{\theta(\tau)<\theta(\tau+1)}} \Big\| \frac{1}{j} \sum_{k=1}^j u_{n_{\theta(k)}}\Big\|_{L^1(X)}\leq \eps. 
\end{align*}
\end{theorem}

\begin{proof}
The weak convergence $(u_n)_n$ is bounded say $\sup_{n\geq1}\|u_n\|_{L^1(X)}\leq r$. By tightness and uniform-integrability(see Theorem \ref{thm:dunford-petit}), consider $X_0\subset X$ with $\mu(X_0)<\infty$ and $\delta>0$  so that  we have 
\begin{align*}
&\sup_{n\geq1}\int_{X\setminus X_0}|u_n(x)|\d\mu(x)\leq \frac{\eps}{6}, \\
&\sup_{n\geq1}\int_{A}|u_n(x)|\d\mu(x)\leq \frac{\eps}{6}\quad \text{whenever $\mu(A)\leq \delta$}. 
\end{align*}
Define $A_{n,m}=\{x\in X_0\,:\, |u_n(x)|\geq m\}$ so that 
\begin{align*}
\sup_{n\geq1}|A_{n,m}|\leq \frac{1}{m} \sup_{n\geq1}\|u_n\|_{L^1(X)}\leq \frac{r}{m}\xrightarrow{m\to\infty}0.
\end{align*}
Next, choose $m_0$ such that $\sup_{n\geq1}|A_{n,m}|\leq \delta$   whenever $m\geq m_0$ so that
\begin{align*}
\sup_{n\geq1}\int_{A_{n,m_0}}|u_n(x)|\d\mu(x)\leq \frac{\eps}{6}. 
\end{align*}
Consider the sequence $(v_n)_n$ defined as follows
\begin{align*}
v_n(x)=
\begin{cases} 
u_n(x)& x\in A_{n,m_0}\cup X\setminus X_0, \\
0& x\in X_0\setminus A_{n,m_0}.
\end{cases}
\end{align*}
On the one hand, the sequence $(v_n)_n$ verifies the estimate 
\begin{align}\label{eq:xbound-vn}
\sup_{n\geq1}\|v_{n}\|_{L^1(X)}
\leq \sup_{n\geq1}\int_{X\setminus X_0}|v_{n}(x)|\d\mu(x) +  \sup_{n\geq1} \int_{A_{n, m_0}}|v_{n}(x)|\d\mu(x)<\frac{\eps}{3}.
\end{align} 
On the other hand, it  follows from the definition of $A_{n,m_0}$ that 
\begin{align*}
\sup_{n\geq1}\int_{X} |u_n(x)-v_n(x)|^2\d\mu(x)=
\sup_{n\geq1}
\int_{X_0\setminus A_{n,m_0}}
|u_n(x)|^2\d\mu(x)\leq \mu(X_0\setminus A_{n,m_0})m_0^2.
\end{align*}
That is the sequence $(u_n-v_n)_n$ is bounded in $L^2(X)$.   In view of Theorem \ref{thm:banach-saks-hilbert}, there is a subsequence $(n_j)_j$ and such that $(u_{n_j}-v_{n_j})_j$ converges  in the weak sense  in $L^2(X)$ to some $w$  and we also have 
\begin{align*}
\lim_{j\to\infty}	\sup_{\stackrel{\theta: \mathbb{N}\to\mathbb{N},\, s.t. }{\theta(\tau)<\theta(\tau+1)}} \Big\| \frac{1}{j} \sum_{k=1}^j (u_{n_{\theta(k)}}-v_{n_{\theta(k)}})-w\Big\|_{L^2(X)}=0.
\end{align*}
From the latter  we find that $(u_{n_{\theta(k)}}-v_{n_{\theta(k)}})=w=0$ a.e. on $X\setminus X_0$ and we deduce 
\begin{align*}
\sup_{\stackrel{\theta: \mathbb{N}\to\mathbb{N},\, s.t. }{\theta(\tau)<\theta(\tau+1)}} \Big\| \frac{1}{j} \sum_{k=1}^j (u_{n_{\theta(k)}}-v_{n_{\theta(k)}})-w\Big\|_{L^1(X)}
\leq \mu(X_0)^{1/2}  \sup_{\stackrel{\theta: \mathbb{N}\to\mathbb{N},\, s.t. }{\theta(\tau)<\theta(\tau+1)}} \Big\| \frac{1}{j} \sum_{k=1}^j (u_{n_{\theta(k)}}-v_{n_{\theta(k)}})-w\Big\|_{L^2(X)}\to0.
\end{align*}
Note that $w\in L^1(X)$ since $\|w\|_{L^1(X)}\leq \mu(X_0)^{1/2} \|w\|_{L^2(X)}$. Therefore, there is $j_0\geq1$ such that 
\begin{align}\label{eq:xx-eps-uniform-vn}
\sup_{\stackrel{\theta: \mathbb{N}\to\mathbb{N},\, s.t. }{\theta(\tau)<\theta(\tau+1)}} \Big\| \frac{1}{j} \sum_{k=1}^j (u_{n_{\theta(k)}}-v_{n_{\theta(k)}})-w\Big\|_{L^1(X)}\leq \frac{\eps}{3}\qquad\text{for every $j\geq j_0$}. 
\end{align}
Furthermore, for $g\in L^{\infty}(X)$ (in particular $g=\mathds{1}_A$, $A\in \mathcal{A}$) we have  $\| \mathds{1}_{X_0}g\|^2_{L^2(X)}\leq \mu(X_0)\|g\|^2_{L^{\infty}(X)}<\infty$ so that  $\mathds{1}_{X_0}g\in L^2(X)$. The weak convergence of the sequence $(u_{n_j}-v_{n_j} )_j$ in $L^2(X)$ implies 
\begin{align*}
\int_{X}(u_{n_{j}}-v_{n_{j}}-w)(x)g(x)\d\mu(x)
&= \int_{X_0}(u_{n_{j}}-v_{n_{j}}-w)(x) g(x)\d\mu(x)\\
&= \int_{X}(u_{n_{j}}-v_{n_{j}}-w)(x) (\mathds{1}_{X_0}g)(x)\d\mu(x)\xrightarrow{j\to\infty}0. 
\end{align*}
Therefore, $u_{n_j}-v_{n_j} \rightharpoonup w$   and $u_{n_j}\rightharpoonup 0$ weakly in  $L^1(X)$ and hence $v_{n_j} \rightharpoonup -w$ weakly in $L^1(X)$. The weak convergence in $L^1(X)$ and the estimate \eqref{eq:xbound-vn} yield  
\begin{align*}
\|w\|_{L^1(X)}= \|w\|_{L^1(X_0)}\leq \liminf_{j\to \infty}\|v_{n_j}\|_{L^1(X)}
<\frac{\eps}{3}.
\end{align*}
Altogether with the estimate \eqref{eq:xx-eps-uniform-vn}, for  every $j\geq j_0$ we arrive at the following estimate
\begin{align*}
\sup_{\stackrel{\theta: \mathbb{N}\to\mathbb{N},\, s.t. }{\theta(\tau)<\theta(\tau+1)}} \Big\| \frac{1}{j} \sum_{k=1}^j u_{n_{\theta(k)}}\Big\|_{L^1(X)}
&\leq 	\sup_{\stackrel{\theta: \mathbb{N}\to\mathbb{N},\, s.t. }{\theta(\tau)<\theta(\tau+1)}} \Big\| \frac{1}{j} \sum_{k=1}^j (u_{n_{\theta(k)}}-v_{n_{\theta(k)}})-w\Big\|_{L^1(X)}
\\&+ \sup_{\stackrel{\theta: \mathbb{N}\to\mathbb{N},\, s.t. }{\theta(\tau)<\theta(\tau+1)}} \frac{1}{j} \sum_{k=1}^j (\| v_{n_{\theta(k)}}\|_{L^1(X)}+ \|w\|_{L^1(X)})<\eps. 
\end{align*}
We finally obtain the sought estimate
\begin{align*}
\limsup_{j\to \infty}	\sup_{\stackrel{\theta: \mathbb{N}\to\mathbb{N},\, s.t. }{\theta(\tau)<\theta(\tau+1)}} \Big\| \frac{1}{j} \sum_{k=1}^j u_{n_{\theta(k)}}\Big\|_{L^1(X)}\leq \eps. 
\end{align*}
\end{proof}

\begin{proof}[Proof of Theorem \ref{thm:uniform-weak-banach-saks}]
We can assume $u=0$. By Theorem \ref{thm:eps-uniform-cesaro}, one  can iteratively  find a nested family of  subsequences $(n_{i,j})_j$,  $i\geq 1$ with the property that $(n_{i+1,j})_j$ is a subsequence of $(n_{i,j})_j$  and there holds 
\begin{align}\label{eq:x-uniform-cesaro}
\sup_{\stackrel{\theta: \mathbb{N}\to\mathbb{N},\, s.t. }{\theta(\tau)<\theta(\tau+1)}} \Big\| \frac{1}{j} \sum_{k=1}^j u_{n_{i, \theta(k)}}\Big\|_{L^1(X)}\leq \frac{1}{i}. 
\end{align}
\noindent In particular if we fix $\ell\geq1$  and a map $\theta:\mathbb{N}\to \mathbb{N}$ strictly increasing, it is clear that for each $k\geq 1$, $(n_{\theta(k+\ell),\theta(j)})_j$ is a subsequence of $(n_{\theta(\ell),j})_j$, namely  there is  $\theta^{\ell}_k:\mathbb{N}\to \mathbb{N}$  strictly increasing such that 
$n_{\theta(k+\ell),\theta(j)}=n_{\theta(\ell),\theta^{\ell}_{k}(j)}$, $j\geq1$.  Observing that  
\begin{align*}
n_{\ell, \theta^{\ell}_{k+1}(j)}
= n_{\theta(k+\ell+1),\theta(j)}
=n_{\theta(k+\ell),\theta^{k+\ell}_{1}(\theta(j))}
= n_{\ell,\theta^{\ell}_{k}\circ\theta^{k+\ell}_{1}(j)}, 
\end{align*}
we can legitimately identify $\theta^{\ell}_{k+1}(j)= \theta^{\ell}_{k}(\theta^{k+\ell}_{k}(j))$. We have  $\theta^{\ell}_{k+1}(j)= \theta^{\ell}_{k}(\theta^{k+\ell}_{1}(j))\geq \theta^{\ell}_{k}(j)$, since $\theta^{k+\ell}_1(j)\geq j$. Hence, taking $j=k+\ell+1$ then as $\theta^{\ell}_{k}$ is strictly increasing,  we get 
\begin{align*}
\theta^{\ell}_{k+1}(k+\ell+1)\geq \theta^{\ell}_{k}(k+\ell+1)>\theta^{\ell}_{k}(k+\ell)
\end{align*}
In other words the mapping $\theta^*:\mathbb{N}\to \mathbb{N}$ with $\theta^{*}(k)=\theta^{\ell}_{k}(k+\ell)$ is strictly increasing. As a result, 
taking into account $ n_{\theta(k+\ell), \theta(k+\ell)}= n_{\theta(\ell),\theta^{\ell}_{k}(k+\ell)} = n_{\theta(\ell),\theta^{*}(k)}$ and the estimate \eqref{eq:x-uniform-cesaro} one deduces the following 
\begin{align*}
\limsup_{j\to \infty}\Big\| \frac{1}{j} \sum_{k=1}^j u_{n_{\theta(k+\ell), \theta(k+\ell)}}\Big\|_{L^1(X)}
&= \limsup_{j\to \infty}\Big\| \frac{1}{j} \sum_{k=1}^j u_{n_{\theta(\ell), \theta^*(k)}}\Big\|_{L^1(X)}\\
&\leq\limsup_{j\to \infty} \sup_{\stackrel{\theta: \mathbb{N}\to\mathbb{N},\, s.t. }{\theta(\tau)<\theta(\tau+1)}} \Big\| \frac{1}{j} \sum_{k=1}^j u_{n_{\theta(\ell), \theta(k)}}\Big\|_{L^1(X)}<\frac{1}{\theta(\ell)}\leq \frac{1}{\ell}, 
\end{align*}
where recall that $\theta(j)\geq j$, $j\geq1$.  On the other hand,  we have 
\begin{align*}
\Big\| \frac{1}{j} \sum_{k=1}^j u_{n_{\theta(k),\theta(k)}}\Big\|_{L^1(X)}
&\leq 
\Big\| \frac{1}{j} \sum_{k=1}^{\ell} u_{n_{\theta(k), \theta(k)}}\Big\|_{L^1(X)}
+\Big\| \frac{1}{j-\ell} \sum_{k=\ell+1}^j u_{n_{\theta(k),\theta(k)}}\Big\|_{L^1(X)} \\
&=\Big\| \frac{1}{j} \sum_{k=1}^{\ell} u_{n_{\theta(k),\theta(k)}}\Big\|_{L^1(X)}
+\Big\| \frac{1}{j-\ell} \sum_{k=1}^{j-\ell}u_{n_{\theta(k+\ell), \theta(k+\ell)}}\Big\|_{L^1(X)}. 
\end{align*}
Using this and the fact that $\sum_{k=1}^\ell\|u_{n_{\theta(k),\theta(k)}}\|_{L^1(X)}\leq \ell \sup_{n\geq1} \|u_{n}\|_{L^1(X)}\leq \ell r$ we get 
\begin{align*}
\limsup_{j\to \infty}  \sup_{\stackrel{\theta: \mathbb{N}\to\mathbb{N},\, s.t. }{\theta(\tau)<\theta(\tau+1)}}\Big\| \frac{1}{j} \sum_{k=1}^j u_{n_{\theta(k),\theta(k)}}\Big\|_{L^1(X)}
&\leq \limsup_{j\to \infty} \sup_{\stackrel{\theta: \mathbb{N}\to\mathbb{N},\, s.t. }{\theta(\tau)<\theta(\tau+1)}}
\Big\| \frac{1}{j-\ell} \sum_{k=1}^{j-\ell}u_{n_{\theta(k+\ell),\theta(k+\ell)}}\Big\|_{L^1(X)} \\
&= \limsup_{j\to \infty} \sup_{\stackrel{\theta: \mathbb{N}\to\mathbb{N},\, s.t. }{\theta(\tau)<\theta(\tau+1)}}\Big\| \frac{1}{j} \sum_{k=1}^{j}u_{n_{\theta(k+\ell),\theta(k+\ell)}}\Big\|_{L^1(X)}\leq  \frac{1}{\ell}. 
\end{align*}
Given that $\ell\geq1$ is arbitrary, letting $\ell\to \infty$ yields 
\begin{align*}
\lim_{j\to \infty} \sup_{\stackrel{\theta: \mathbb{N}\to\mathbb{N},\, s.t. }{\theta(\tau)<\theta(\tau+1)}} \Big\| \frac{1}{j} \sum_{k=1}^j u_{n_{\theta(k),\theta(k)}}\Big\|_{L^1(X)}= \limsup_{j\to \infty}\sup_{\stackrel{\theta: \mathbb{N}\to\mathbb{N},\, s.t. }{\theta(\tau)<\theta(\tau+1)}}\Big\| \frac{1}{j} \sum_{k=1}^j u_{n_{\theta(k),\theta(k)}}\Big\|_{L^1(X)}=0. 
\end{align*}
Finally  the desired subsequence $(u_{n_j})_j$ is obtained by taking that diagonal sequence $ u_{n_j}= u_{n_{j,j}}$. 
\end{proof}

\appendix
\section{} \label{sec:Appendix}
\begin{theorem}[S. Kakutani \cite{Kak39}] \label{thm:banach-saks-kakut}
A uniformly convex Banach  space $X$ has Banach-Saks property.
\end{theorem}

\noindent The result by P. Enflo \cite[Corollary 4]{Enf72} implies  that a Banach  space  is uniformly convexfiable\footnote{A Banach space is uniformly convexifiable if it can be equipped with an equivalent norm that renders it uniformly convex.} if and only if its dual is uniformly convexifiable. Thus Theorem \ref{thm:banach-saks-kakut} can be recast as the following Theorem \ref{thm:banach-saks-Okada}. 
\begin{theorem}[N. Okada \cite{Oka84}] \label{thm:banach-saks-Okada}
A Banach  space $X$ whose dual $X'$ is uniformly convex has the Banach-Saks property.
\end{theorem}

\noindent We present an elegant proof of Theorem \ref{thm:banach-saks-Okada} due to N. Okada \cite{Oka84} based on the duality mapping of $X$, viz., the map $\varphi:X\to X'$  verifying $\varphi(0)=0$ and 
\begin{align*}
(\varphi(x), x)= \|x\|_X \|\varphi(x)\|_{X'}=\|x\|^2_X, 
\end{align*}
where $(\cdot, \cdot)$ is the dual pairing of $X$ and $X'$.  The existence of $\varphi(x)\in X'$ is a consequence of the Hahn-Banach theorem, whereas  since $X'$ is uniformly convex, it can be readily shown that the uniqueness follows from the strict convexity of $X'$. 
Furthermore, the uniqueness implies that $\varphi(\lambda x)= \lambda\varphi(x)$, $\lambda\in \R$, while uniform convexity of $X'$ implies that $\varphi: X\to X'$ is uniformly continuous on bounded sets; see for instance \cite[Section 5.4]{Chi09} or  \cite{Kat67}.
\begin{proof}
As $X'$ uniform convex implies $X'$ is reflexive which is equivalent to says $X$. I	t suffices to prove the weak Banach-Saks property.
A sequence $(x_n)_n\subset $  weakly converging to $x$ in $X$ is bounded, say $(x_n)_n\subset B_X(0,r):=\{x\in X\,:\, \|x\|_X\leq r\}$ for some $r>0$. Without loss of generality, assume $x=0$. Put $x_{n_1}=x_1$ assume $x_{n_{j-1}}$ is given, $j\geq2$, and put $S_{j-1}=\sum_{k=1}^{j-1} x_{n_{k}}.$ 
Since $ \varphi(S_{j-1})\in X' $ and $(x_n)_n$ converges weakly to $x=0$, we choose $n_j>n_{j-1}$ and hence construct the sequence $(x_{n_j})_j$ such that 	$|(\varphi(S_{j-1}), x_{n_j})|\leq 1$.  Since $\varphi$ is uniformly continuous on $B_X(0,r)$ for $\eps>0$  there is $\delta>0$ such that $\|\varphi(x)-\varphi(y)\|_{X'}<\eps/r$ wherever $x,y\in B_X(0,r)$ and $\|x-y\|_X\leq \delta$. Fix
$j>j_0$ with $j_0\geq r/\delta,$ then $	\frac{1}{j}\|S_{j}-S_{j-1}\|_{X}\leq \frac{r}{j_0}\leq \delta$. Whence,  we have 
\begin{align*}
\big|\big( \varphi\big(\frac{S_{j}}{j}\big)-\varphi\big(\frac{S_{j-1}}{j}\big), S_{j} -S_{j-1}\big)\big|
&\leq \|x_{n_j}\|_{X}\big\|\varphi\big(\frac{S_{j}}{j}\big)
-\varphi\big(\frac{S_{j-1}}{j}\big)\big\|_{X'}\leq \frac{r\eps}{r}=\eps. 
\end{align*}
That is, using the formula $\varphi(\lambda x)= \lambda\varphi(x)$ we obtain  
\begin{align*}
\big|\big(\varphi(S_{j})-\varphi(S_{j-1}),S_{j} -S_{j-1}\big)\big|
\leq j\eps. 
\end{align*}
The relations $(\varphi(x), x)= \|x\|_X \|\varphi(x)\|_{X'}$ and $\|x\|_X= \|\varphi(x)\|_{X'}$ yield
\begin{align*}
\big(\varphi(x)-\varphi(y),x-y\big)
&=\big(\|x \|_X -\|y \|_{X}\big)^2_X+ \big[\|x \|_X \|y \|_{X}-(\varphi(x), y)\big]+  \big[\|x\|_X \|y \|_{X}-(\varphi(y),x)\big]. 
\end{align*}
Taking into account the fact that, $(\varphi(x), y)\leq\|\varphi(x)\|_{X'}\|y\|_X= \|x\|_X\|y\|_X$, it follows that  
\begin{align*}
\big( \varphi(x)-\varphi(y), x-y\big)
\geq  \big[\|x\|_X \|y \|_{X}-(\varphi(y),x)\big]\geq0. 
\end{align*}
Accordingly, since $S_{j}-S_{j-1}=x_{n_j}$ we find that 
\begin{align*}
0\leq\|S_{j} \|_X \|S_{j-1} \|_{X} - \|S_{j-1} \|^2_{X}-(\varphi(S_{j-1}), x_{n_j})
&\leq \big(\varphi(S_{j})-\varphi(S_{j-1}), S_{j} -S_{j-1}\big)\leq j\eps.
\end{align*}
Given that $(\varphi(S_{j-1}), x_{n_j})\leq 1$, by definition of $x_{n_j}$ we get 
\begin{align*}
\|S_{j} \|_X \|S_{j-1} \|_{X} - \|S_{j-1} \|^2_{X}
\leq j\eps+ (\varphi(S_{j-1}), x_{n_j})\leq j\eps+1. 
\end{align*}
It follows that,  for all $j>j_0$ we have 
\begin{align*}
\|S_{j} \|^2_X  - \|S_{j-1} \|^2_{X}&= \big(	\|S_{j} \|_X - \|S_{j-1} \|_{X}\big)^2+ 2	\big( \|S_{j} \|_X \|S_{j-1} \|_{X} - \|S_{j-1} \|^2_{X}\big)\\
&\leq  \|x_{n_j} \|^2_{X}+ 2	( j\eps+1)\leq r^2+ 2( j\eps+1) . 
\end{align*}
Whence, summing both side gives 
\begin{align*}
\|S_{j} \|^2_X  - \|S_{j_0} \|^2_{X}= \sum_{k=j_0+1}^j \|S_{k} \|^2_X  - \|S_{k-1} \|^2_{X}
\leq jr^2+ 2j( j\eps+1).  
\end{align*}
This implies that 
\begin{align*}
\limsup_{j\to \infty}\frac{1}{j^2}	\|S_{j} \|^2_X \leq\limsup_{j\to \infty} \frac{1}{j^2} \|S_{j_0} \|^2_{X}+ \frac{r^2}{j} + 2\eps+\frac{2}{j}\leq 2\eps. 
\end{align*}
Finally, as $\eps>0$ is arbitrarily chosen we get 
\begin{align*}
\lim_{j\to \infty}		\|\frac{S_{j}}{j} \|^2_X = 	\limsup_{j\to \infty}		\|\frac{S_{j}}{j} \|^2_X =0. 
\end{align*}
\end{proof}

\noindent It is well-known that \cite{Cla36} both $L^p(X)$ and its dual space $(L^{p}(X))'\equiv L^{p'}(X)$ with $p'=\frac{p}{p-1}$ are uniformly convex when $1<p<\infty$. A  proof of the uniform convexity of $L^p(X)$, $1<p<\infty$ can be found in  \cite[Theorem 5.4.2.]{Wil13} or see also \cite{Shi18} for a short proof. Hence both Theorem \ref{thm:banach-saks-kakut} and Theorem \ref{thm:banach-saks-Okada} imply that  $L^p(X)$, $1<p<\infty$ satisfies the Banach-Saks property. However we present here a simple proof due to Banach and Saks. 

\begin{theorem}[{Banach-Saks \cite{BaSa30}}] \label{thm:banach-saks-lp} The space $L^p(X)$, $1<p<\infty$  satisfies the Banach-Saks property. 
\end{theorem}
\begin{proof}
Since $L^p(X)$ is reflexive, it suffices to prove the weak Banach-Saks property. Let  $(u_n)_n\subset L^p(X)$  a bounded sequence, say $\sup_{n\geq1}\|u_n\|_{L^p(X)}\leq r$, $r>0$, weakly converging  to $u\in L^p(X)$ that is 
\begin{align*}
\int_X(u_n(x)-u(x))v(x)\d\mu(x) \xrightarrow{n\to\infty}0 \qquad\text{for all $v\in L^{p'}(X)$}. 
\end{align*}
Without loss of generality, assume $u=0$. Put $u_{n_1}=x_1$ assume $u_{n_{j-1}}$ is given, $j\geq2$. 
Define $S_{j-1}=\sum_{k=1}^{j-1} u_{n_{k}}\in L^p(X)$ so  that $ | S_{j-1}|^{p-2} S_{j-1}\in L^{p'}(X) $. Since $(u_n)_n$ converges weakly to $u=0$, we choose $n_j>n_{j-1}$ and hence construct the sequence $(u_{n_j})_j$ such that 
\begin{align*}
\int_{X}| S_{j-1}(x)|^{p-2} S_{j-1}(x)  u_{n_j}(x)\d\mu(x) \leq 1.
\end{align*}
Next, consider the  continuous function $\zeta:\R\to \R$, 
\begin{align*}
\zeta(t)=\frac{|1+t|^p-  \sum_{k=0}^{\lfloor p\rfloor}\binom{p}{k}t^{k} }{|t|^p}
\end{align*}
where $\lfloor p\rfloor=\max\{m\in\mathbb{N}\,:\, m\leq p\}$ and $\binom{p}{k}=\frac{\Gamma(p+1)}{k!\Gamma(p-k+1)}$. Note that, if $\lfloor p\rfloor<p$ then $\zeta(t)\xrightarrow{|t|\to\infty}1$  and if $p=\lfloor p\rfloor$  then $\zeta(t)\xrightarrow{|t|\to\infty}0$,   
whereas,  using Taylor's expansion for $|t|<1$ we deduce
\begin{align*}
\zeta(t)=\frac{(1+t)^p-  \sum_{k=0}^{\lfloor p\rfloor}\binom{p}{k}t^{k} }{|t|^p}
= \frac{\sum_{k=\lfloor p\rfloor+1}^\infty\binom{p}{k}t^{k} }{|t|^p}\xrightarrow{|t|\to0}0.
\end{align*}
Therefore, the map $t\mapsto\zeta(t)$ is bounded say $|\zeta(t)|\leq C$ for some $C>1$, yielding 
\begin{align*}
|1+t|^p\leq C|t|^p+ \sum_{k=0}^{\lfloor p\rfloor}\binom{p}{k}t^{k}.
\end{align*}
This implies that for all $a,b\in\R$,  we have 
\begin{align*}
|a+b|^p\leq
\begin{cases}
|a|^p+ pb|a|^{p-2}a+C|b|^p+ \sum_{k=2}^{\lfloor p\rfloor}\binom{p}{k}|a|^{p-k}|b|^{k} & \text{if}\,\, p\geq2\\
|a|^p+ pb|a|^{p-2}a+C|b|^p &\text{if}\,\, 1<p<2.
\end{cases} 
\end{align*}
The foregoing inequality with $a= S_{j-1}$ and $b=u_{n_j}$ yields 
\begin{align*}
\| S_j\|^p_{L^p(X)}-\| S_{j-1}\|^p_{L^p(X)}
& \leq  p\int_{X} | S_{j-1}(x)|^{p-2} S_{j-1}(x) u_{n_j} (x)\d \mu(x)+ C\|u_{n_j}\|^p_{L^p(X)} +R_j\\
&\leq p +Cr^p+ R_j,
\end{align*}
where we define the remainder 
\begin{align*}
R_j=
\begin{cases}
\sum_{k=2}^{\lfloor p\rfloor}\binom{p}{k} \int_{X} |S_{j-1}(x)|^{p-k}|u_{n_j}(x)|^{k}\d\mu(x)& \text{if}\,\, p\geq2\\
	0&\text{if}\, \, 1<p<2.
\end{cases} 
\end{align*}
In  such a way that 
\begin{align}\label{eq:xxRk-estimate}
\| S_j\|^p_{L^p(X)}-\| S_{1}\|^p_{L^p(X)} =  \sum_{k=2}^j	\| S_k\|^p_{L^p(X)}-\| S_{k-1}\|^p_{L^p(X)}
&< (j-1)(p +Cr^p)+ \sum_{k=2}^j R_k. 
\end{align}
Using $\|u_n\|_{L^p(X)}\leq r$ and H\"older inequality, we obtain  for $2\leq k\leq p$, 
\begin{align*}
\int_{X} |S_{j-1}(x)|^{p-k}|u_{n_j}(x)|^{k}\d\mu(x)
&\leq  \|u_{n_j}\|^k_{L^p(X)} \big\|\sum_{i=1}^{j-1} u_{n_{i}} \big\|^{p-k}_{L^p(X)}
< j^{p-2}r^p. 
\end{align*}
Summing up both sides gives, for $p\geq2$,
\begin{align*}
R_j= \sum_{k=2}^{\lfloor p\rfloor}\binom{p}{k} \int_{X} |S_{j-1}(x)|^{p-k}|u_{n_j}(x)|^{k}\d\mu(x)
\leq \sum_{k=2}^{\lfloor p\rfloor}\binom{p}{k}j^{p-2}r^p = Bj^{p-2}r^p, 
\end{align*} 
wherefrom we deduce that $\sum_{k=2}^j R_k< Br^pj^{p-1}$ with $B=\sum_{k=2}^{\lfloor p\rfloor}\binom{p}{k}$. Inserting this in   \eqref{eq:xxRk-estimate} gives 
\begin{align*}
\| \frac{1}{j}S_j\|^p_{L^p(X)}
&< \| \frac{1}{j}S_{1}\|^p_{L^p(X)} +\frac{j-1}{j^p}(p +Cr^p)+ \sum_{k=2}^j R_k\\
&<\begin{cases}
\frac{r^p}{j^p}+\frac{1}{j^{p-1}}(p +Cr^p)+  \frac{1}{j}Br^p &\text{if}\,\, p\geq 2\\
	\frac{r^p}{j^p}+ \frac{1}{j^{p-1}}(p +Cr^p)&\text{if}\,\, 1<p<2. 
\end{cases}
\end{align*}
Whence, this implies that $\lim_{j\to\infty} \| \frac{1}{j}S_j\|_{L^p(X)}=0.$
\end{proof}

\vspace{-2mm}
\bibliographystyle{alpha}

\end{document}